\documentclass{amsart}
\usepackage{amsfonts}
\usepackage{color}

\usepackage{amsmath,amssymb,amsthm, epsfig}

\usepackage{hyperref}

\usepackage{amsmath}

\usepackage{amssymb}

\newcommand{\lam}{\lambda}

\def\R{{\mathbb {R}}}

\def\lam{\lambda}

\def\ve{\varepsilon}

\newlength{\hchng}
\newlength{\vchng}
\def \R {\mathbb{R}}

\def \div {\mathrm{div}}

\def \dist {\mathrm{dist}}

\def \suchthat {\ \big | \ }

\def \ve {\varepsilon}

\newcommand{\defeq}{\mathrel{\mathop:}=}

\newtheorem{theorem}{Theorem}[section]
\newtheorem{lemma}[theorem]{Lemma}

\newtheorem{corollary}[theorem]{Corollary}
\theoremstyle{definition}

\newtheorem{definition}[theorem]{Definition}
\newtheorem{example}[theorem]{Example}
\theoremstyle{remark}
\newtheorem{remark}[theorem]{Remark}
\numberwithin{equation}{section}

\newcommand{\intav}[1]{\mathchoice {\mathop{\vrule width 6pt height 3 pt depth  -2.5pt
\kern -8pt \intop}\nolimits_{\kern -6pt#1}} {\mathop{\vrule width
5pt height 3  pt depth -2.6pt \kern -6pt \intop}\nolimits_{#1}}
{\mathop{\vrule width 5pt height 3 pt depth -2.6pt \kern -6pt
\intop}\nolimits_{#1}} {\mathop{\vrule width 5pt height 3 pt depth
-2.6pt \kern -6pt \intop}\nolimits_{#1}}}

\begin{document}
	
\title[Sharp regularity estimates for quasi-linear elliptic dead core problems]{Sharp regularity estimates for quasi-linear elliptic dead core problems and applications}

\author[J.V. da Silva and A.M. Salort]{Jo\~{a}o V\'{i}tor da Silva and Ariel M. Salort}

\address{Departamento de Matem\'atica, FCEyN - Universidad de Buenos Aires and
\hfill\break \indent IMAS - CONICET
\hfill\break \indent Ciudad Universitaria, Pabell\'on I (1428) Av. Cantilo s/n. \hfill\break \indent Buenos Aires, Argentina.}

\email[J. V. da Silva]{jdasilva@dm.uba.ar}

\email[A.M. Salort]{asalort@dm.uba.ar}
\urladdr{http://mate.dm.uba.ar/~asalort}

\begin{abstract}
In this manuscript we study geometric regularity estimates for quasi-linear elliptic equations of $p$-Laplace type ($1 < p< \infty$) with strong absorption condition:
\begin{equation*}
      -\div(\Phi(x, u,  \nabla u)) + \lambda_0(x) u_{+}^q(x) = 0 \quad \mbox{in} \quad \Omega \subset \R^N,
\end{equation*}
where $\Phi: \Omega \times \R_{+} \times \R^N \to \R^N$ is a vector field with an appropriate $p$-structure,  $\lambda_0$ is a non-negative and bounded function and $0\leq q<p-1$. Such a model is mathematically relevant because permits existence of solutions with dead core zones, i.e, \textit{a priori} unknown regions where non-negative solutions vanish identically. We establish sharp and improved $C^{\gamma}$ regularity estimates along free boundary points, namely $\mathfrak{F}_0(u, \Omega) = \partial \{u>0\} \cap \Omega$, where the regularity exponent is given explicitly by $\gamma = \frac{p}{p-1-q} \gg 1$. Some weak geometric and measure theoretical properties as non-degeneracy, uniform positive density and porosity of free boundary are proved. As an application, a Liouville-type result for entire solutions is established provided that their growth at infinity can be controlled in an appropriate manner. Finally, we obtain finiteness of $(N-1)$-Hausdorff measure of free boundary for a particular class of dead core problems. The approach employed in this article is novel even to dead core problems governed by the $p$-Laplace operator $-\Delta_p u + \lambda_0 u^q\chi_{\{u>0\}} = 0$ for any $\lambda_0>0$.
\end{abstract}
	
\maketitle

\section{Introduction}

Quasi-linear elliptic equations whose nonlinear nature give rise to free boundaries come from as varied  phenomena as reaction-diffusion and absorption processes in pure and applied mathematics. Throughout the last decades we can find in the literature some remarkable examples which model several problems derived from theory of chemical-biological processes, combustion phenomenon and population dynamics, just to mention a few. Regarding this studies, an often more relevant problem to be dealt from a applied point of view is that which arises from diffusion processes with sign constrain, which are in general the only significant cases in physical situations (cf. \cite{Aris1}, \cite{Aris2}, \cite{BSS}, \cite{Diaz} and \cite{HM} for some motivational works). An interesting example is given by
\begin{equation}\label{DCP}
\left\{
\begin{array}{rclcl}
     -\Delta_p u(x) + \lambda_0(x) f(u)\chi_{\{u>0\}}& = & 0 & \mbox{in} & \Omega \\
     u(x) & = & g(x) & \mbox{on} & \partial \Omega,
\end{array}
\right.
\end{equation}
where $1<p< \infty$, $\Omega \subset \R^N$ is a regular and bounded domain, $\Delta_p u(x) = \div(|\nabla u|^{p-2}\nabla u)$ is the well-known $p$-Laplace operator and $\lambda_0 $ is a positive bounded function. In such a context $\lambda_0$ is known as \textit{Thiele Modulus} and it controls the ratio of reaction rate to diffusion-convection rate. Here $f$ is a continuous and increasing reaction term satisfying $f(0) \geq 0$ and $g$ is a continuous non-negative boundary value datum (in applied sciences, $f$ represents the ratio of reaction rate at concentration $u$ to reaction rate at concentration unity). When the nonlinearity $f \in C^1(\Omega)$ is locally $(p-1)-$Lipschitz near zero\footnote{We said that $f$ satisfies a Lipschitz condition of order $p-1$ at $0$ if there exist constants $\mathfrak{M}, \delta>0$ such that $f(u)\leq \mathfrak{M}u^{p-1}$ for $0<u<\delta$.}, it follows from the Maximum Principle that nonnegative solutions must be, in fact, strictly positive (cf. \cite{PS} and \cite{Vaz}). However, the function $f$ may fail to be differentiable or even not decaying fast enough at origin. For instance, if $f(t) \approx t^q$ with $0<q<p-1$, $f$ fails to be Lipschitz of order $p-1$ at the origin; in this case,  problem \eqref{DCP} has an absence of Strong Minimum Principle, i.e., non-negative solutions may vanish completely within an \textit{a priori} unknown region of positive measure $\Omega^{\prime} \subset \Omega$ known as \textit{dead core} set (cf. D\'{i}az's Monograph \cite[Chapter 1]{Diaz} for a complete survey about this subject of research). Such a peculiar characteristic of dead-core solutions allow us to treat \eqref{DCP} as a free boundary problem.

This class of dead core free boundary problems has received warm attention since the late 70's. A huge amount of investigations were carried out on  this topic of research, including existence of solutions and dead core sets, properties of localization, asymptotic behavior of solutions and  ``effectiveness factor'' among others, see, for instance the works due to Bandle \textit{et al} \cite{BSS, BV}, D\'{i}az \textit{et al} \cite{Diaz79}, \cite{Diaz}, \cite{DiazHern}, \cite{DiazHerr}, \cite{DiazVer} and Pucci-Serrin \cite{PS, PS06}. Despite of the large literature on divergence form dead core problems, quantitative properties for models with non-uniformly elliptic characters and a general structure are far less studied (cf. da Silva \textit{et al} \cite{OSS} and \cite{daSRS} as example of such considerations), and this has been our main impetus for the studies in the current article.

Therefore, in this article we study diffusion problems governed by quasi-linear elliptic equations of $p$-Laplace type for which a Minimum Principle is not available:
\begin{equation}\label{DCP1}
\left\{
\begin{array}{rclcc}
  -\div(\Phi(x, u, \nabla u)) + \lambda_0(x).f(u)\chi_{\{u>0\}}& = & 0 & \mbox{in} & \Omega \\
  u(x) & = & g(x) & \mbox{on} & \partial \Omega,
\end{array}
\right.
\end{equation}
where $\Phi: \Omega \times \R_{+} \times \R^N \to \R^N$ satisfies respectively a $p$-ellipticity and $p$-growth condition, which will be specified soon, $0\leq g \in C^0(\partial \Omega)$, $\lambda_0 \in C^0(\overline{\Omega})$ is a non-negative bounded function and $f$ is a continuous and increasing function with $f(0) = 0$. We are particularly interested in prototypes coming from combustion problems, chemical models (porous catalysis) or enzymatic processes where the existence of dead cores plays an important role in the model (cf. \cite{Aris1}, \cite{Aris2} and \cite{HM} for more explanations). For example, when $u$ represents the density (or temperature) of a chemical reagent (or gas), where such a solution is vanishing, it delineates a region where no reagent (temperature) is present. The standard model is given by
\begin{equation}\label{Maineq}
     -\div(\Phi(x, u, \nabla u)) + \lambda_0(x).u^{q}\chi_{\{u>0\}} = 0 \quad \mbox{in} \quad \Omega,
\end{equation}
where $0\leq q < p-1$ is called the \textit{order of reaction}, and \eqref{Maineq} is said to be an equation with \textit{strong absorption} condition.

 The study of \eqref{Maineq} is meaningful, not only for its applications, but also for its innate relation with several free boundary problems appearing in the literature  (cf. Alt-Phillips \cite{AP}, Andersson \cite{And}, da Silva \textit{et al} \cite{daSRS}, D\'{i}az \cite{Diaz}, Friedman-Phillips \cite{FriePhil}, Leit\~{a}o-Teixeira \cite{LT15} and Phillips \cite{Phil} for a variational treatment and da Silva \textit{et al} \cite{LRS} and Teixeira \cite{Tei-16} for a non-variational counterpart; we refer also to reader da Silva \textit{et al} \cite{OS} and \cite{OSS} for problems in the parabolic setting).

\subsection{Main hypothesis and overview of article}

Given a bounded domain $\Omega \subset \R^N$ we consider a function $\Phi: \Omega\times\R\times \R^N\to \R^N$ satisfying the following structural properties:
\vspace{0.2cm}

\begin{itemize}
	\item[(H1)][{\bf Continuity}]. $\Phi \in C^0(\Omega\times\R\times \R^N; \R^N)$.

	\item[(H2)][{\bf Monotonicity}]. For every $\xi_1,\xi_2\in \R^n$ and $(x, s) \in \Omega \times \R$ there holds
$$
   \langle \Phi(x,z,\xi_1)-\Phi(x,z,\xi_2), \xi_1-\xi_2 \rangle \geq 0.
$$

	\item[(H3)][{\bf $p$-Ellipticity}]. There exists positive constants $c_1, c_2$ such that for all $(x,z,\xi)\in\Omega\times \R\times \R^N$ there holds
	$$
		\langle \Phi(x,z,\xi),\xi\rangle \geq c_1|\xi|^p - c_2 |z|^p
	$$
	where $p\in (1,\infty)$.

	\item[(H4)][{\bf $p$-Growth}]. There exists positive constants $c_3, c_4$ such that for all $(x,z,\xi)\in\Omega\times \R\times \R^N$ and $1<p<\infty$ there holds
	$$
		|\Phi(x,z,\xi)| \leq c_3|\xi|^{p-1}+ c_4|z|^{p-1}
	$$

\end{itemize}

\begin{example}
A prototypical example for such a family of operators is given by
$$
	\Phi(x, z, \xi)= \mathfrak{A}(x)|\xi|^{p-2}\xi + \mathfrak{b}(x)|z|^{p-2}z
$$
where $\mathfrak{A}\in \R^{N\times N}$ is a symmetric uniformly elliptic and bounded matrix and $\mathfrak{b}$ is a non-negative bounded function. Another class of examples consist of
$$
  \Phi(x, z, \xi)= |\mathfrak{A}(x)\xi\cdot \xi|^{\frac{p-2}{2}}\mathfrak{A}(x)\xi + \mathfrak{b}(x)|z|^{p-2}z,
$$
where $\mathfrak{A}(x)$ and $\mathfrak{b}$ are as above. Notice that in the previous examples when $\mathfrak{A}(x) = I_{N}$ (the identity matrix) and $\mathfrak{b} \equiv 0$, we recover the $p$-Laplace operator.

We can refer (provided that $\Phi(\cdot, \cdot, \xi)$ is restrict to compact sets) the $p$-generalized mean curvature operator, namely
$$
  \Phi(x, z, \xi)= \frac{\mathfrak{A}(x)|\xi|^{p-2}}{(\sqrt{1+|\xi|^2})^m}\xi + \mathfrak{b}(x)|z|^{p-2}z,
$$
for any $m>0$ and $\mathfrak{A}(x)$ and $\mathfrak{b}$ as before. Observe that we recover the mean curvature operator when $p=2$ and $m=1$.
Finally, provided that $\Phi(\cdot, \cdot, \xi)$ is restrict to compact sets, we can cite  the following class of operators
$$
  \Phi(x, z, \xi)= \frac{\mathfrak{A}(x)|\xi|^{p-2}}{(1+|\xi|^p)^{1-\frac{s}{p}}}\xi + \mathfrak{b}(x)|z|^{p-2}z,
$$
where $1<s\leq p< \infty$. Such an operator reduces to $p$-Laplace when $s=p$, $\mathfrak{A}(x) = I_{N}$ and $\mathfrak{b} \equiv 0$.
\end{example}

Finally, by way of motivation, fix $0<q<p-1$ and $R>r_0>0$. Then, the radial profile $u: B_R(0) \to \R_{+}$ given by
$$
    u(x) = (|x|-r_0)_{+}^{\frac{p}{p-1-q}}
$$
is a weak solution to
$$
     -\Delta_p u + \mathfrak{c}(N, p, q).u^{q}(x)= 0 \quad \mbox{in} \quad B_R(0).
$$
A feature worth commenting about this example is that, in general, solutions to \eqref{Maineq} are know to be locally of class $C^{1, \gamma}$ for some $\gamma \in (0, 1)$ (cf. \cite{Choe1}, \cite{DB0}, \cite{LU} and \cite{Tolk}).  However, for such an example, fixed $p>1$, one observes that $u \in C_{\text{loc}}^{\lfloor \alpha \rfloor, \beta}(B_R(0))$, where
$$
  \alpha(p, q) \defeq \frac{p}{p-1-q} \quad \text{and} \quad \beta(p, q) \defeq \frac{p}{p-1-q} - \left\lfloor \frac{p}{p-1-q} \right\rfloor.
$$
Moreover, notice that $\alpha(p, q)> 2$ provided that $q > \max\left\{0,  \frac{p-2}{2}\right\}$, which means that this is a classical solution, even across the free boundary. The proper understanding of such a phenomenon should yield decisive geometric information about the solution and its free boundary, and this is the main goal of our investigation, which focus on a systematic and non-linear approach for such a class of problems. Therefore,  we will show that \textit{any} solution to \eqref{Maineq} behaves near its free boundary like the previous example.

Precisely, we will prove the following improved regularity estimate at free boundary points:

\begin{theorem}\label{MThm} Let $u$ be a bounded weak solution to \eqref{Maineq}. Then, given $\tau >0$ there exists a positive universal constant\footnote{Throughout this manuscript, we will refer to {\it universal constants} when they depend only on dimension and structural properties of the problem, i.e. on $N, p, q, c_1, c_2, c_3, c_4$ and the bounds of $\lambda_0$} $\mathfrak{C}  = \mathfrak{C} (N, p, q, \tau, \lam_+)$ such that for any $x_0 \in \Omega$ fulfilling $B_{\tau}(x_0) \subset \Omega$ and any $r \leq \frac{\tau}{2}$, the following estimate holds:
\begin{equation}\label{estim}
  \displaystyle \sup_{B_r(x_0)} u(x) \leq \mathfrak{C}  \max\left\{\inf_{B_r(x_0)} u(x), \,\, r^{\frac{p}{p-1-q}} \right\}.
\end{equation}
Particularly, if $x_0 \in \partial\{u>0\}\cap \Omega$ (a free boundary point), then
$$
   \displaystyle \sup_{B_r(x_0)} u(x)  \leq \mathfrak{C} r^{\frac{p}{p-1-q}}
$$
for all $0< r< \min\{1, \frac{\dist(x_0, \partial \Omega)}{2}\}$.
\end{theorem}

One more time we must stress the comparison between the regularity coming from dead core solutions and those coming from the classical elliptic Schauder regularity theory. For didactic reasons, let us suppose that $u$ is a weak solution to
\begin{equation}\label{eqDCvsST}
    -\div(\nabla u) + \lambda_0(x)u^{q}(x)\chi_{\{u>0\}} = 0 \quad \text{in} \quad B_1,
\end{equation}
where $0<q<1$ and $\lambda_0 \in C^{0, q}(B_1)$.

Under such assumptions the Schauder regularity theory assures that classical solutions to \eqref{eqDCvsST} are $C_{\text{loc}}^{2, q}(B_1)$ (particularly at free boundary points). On the other hand, our main Theorem  \ref{MThm} claims that $u$ is $C^{\alpha, \beta}$ at free boundary points, where
$$
  \alpha \defeq \left\lfloor \frac{2}{1-q}\right\rfloor \quad \text{and} \quad \beta \defeq \frac{2}{1-q} - \left\lfloor\frac{2}{1-q}\right\rfloor.
$$
Nevertheless, we must highlight that for any $0<q<1$ one have
$$
  \frac{2}{1-q} > 2+q,
$$
which means that dead core solutions are more regular, along free boundary points, that the available  best regularity result coming from classical regularity theory.

By way of motivation, we must highlight that \eqref{Maineq} (provided $\Phi$ does not depend on $u$) can be understood as the Euler-Lagrange equation related to the the following nonlinear minimization problem
\begin{equation} \label{vari}
	\min \int_\Omega \mathcal{F}(x, u, \nabla u) \, dx,
\end{equation}
where the minimum is taken over all non-negative functions in $W^{1,p}(\Omega)$ such that $u=g$ on $\partial \Omega$ and
$$
   \mathcal{F}(x, u, \nabla u) = \mathcal{G}(x, \nabla u) + \mathfrak{h}(x, u),
$$
for suitable $\mathcal{G}$ and $\mathfrak{h}$, see Section \ref{Sec6} for more details.
Recall that the study of the variational problem like \eqref{vari} is fairly developed and well-understood currently, see Manfredi \cite{Manf88} for a survey on this subject. As another example, we can cite Leit\~{a}o-Teixeira \cite{LT15}, where it is studied this minimization problem and established several analytic and geometric properties for weak solutions to \eqref{Maineq}.

We must highlight that the approach leading our results (regularity and qualitative/quantitative analysis) are novelties in the literature and differs from the techniques used in \cite{LRS}, \cite{OS}, \cite{OSS}, \cite[Chapter 1]{Diaz}, \cite[Section 4]{Tei16} and \cite{Tei-16}, where different dead core type problems for divergence and non-divergence form operators were studied. Moreover, our results extend, in some extent,  the previous works for obstacle problems with zero constraint, namely, when $q=0$ (see Shahgholian \textit{et al} \cite{KKPS} and \cite{LeeShah} for some examples on this subject). Finally, we also lead to interesting results in the $p$-Laplace setting, i.e., as $\Phi(x, u, \nabla u) = |\nabla u|^{p-2}\nabla u$ (compare with \cite{daSRS}).

 In contrast with the classical obstacle problem for the Laplace operator with zero constraint, which admits $C^{1, 1}$ solutions, the passage to the general quasi-linear heterogeneous counterpart carries several difficulty levels. The first one is the nonlinear character of the problem. The second one is the lack of homogeneity for such a general quasi-linear elliptic operators. Another pivotal feature of dead core problems consists in that $|\nabla u|$ vanishes at free boundary points, i.e. the operator becomes degenerate/singular along the free boundary. For this very reason, controlling solutions near the free boundary is a non-trivial task. Finally, we must deal with a more complicated form of the Harnack's inequality, as well as \textit{a priori} gradient estimates, which are few develop in the literature.

In conclusion our paper is organized as follows: In Section \ref{Section2} we present the structural properties of the operators that we treat throughout the article. Yet in Section \ref{Section2}, we deliver some results about quasi-linear elliptic problems which are useful in our studies. In Section \ref{Sec3}, we deliver a proof of Theorem \ref{MThm} and its consequences. Section \ref{Sec3} is devoted to  analyse the borderline case, i.e., as $q=p-1$, where we prove that solutions cannot vanish at interior points  unless they are identically zero. Section \ref{Sec4} is devoted to prove some weak geometric properties such as non-degeneracy. More precisely, a dead core solution $u$ leaves the free boundary precisely as
$$
  u(x) \geq \dist(x_0, \partial \{u > 0\})^{\frac{p}{p-1-q}} \quad \forall \,\,\, x \in \overline{\{u>0\}}\cap \Omega.
$$
As consequence from this pivotal geometric information, we obtain positive density and porosity of the free boundary. Section \ref{Sec5} is dedicated to applications of the main results, such as Liouville type results. Our Liouville Theorem assures that
$$
  u(x) \geq \Theta(\lambda_0, N, p, q)|x|^{\frac{p}{p-1-q}} \quad \forall \,\,\,|x| \gg 1
$$
for a suitable constant $\Theta(\lambda_0, N, p, q)>0$ unless that $u$ be identically zero. Finally, in Section \ref{Sec6}, we study the finiteness of $(N-1)$-Hausdorff measure estimates of free boundary. Precisely, we will show that for a particular class of problems modelled by \eqref{Maineq} we have that
$$
   \mu_{\Phi}  = \div(\Phi(x, \nabla u)) - (q+1)\lambda_0(x)u_{+}^{q}\chi_{\{u>0\}}
$$
defines a non-negative Radon measure supported along the free boundary. Consequently, for such class, we prove that
$$
  \mathfrak{c}_0r^{N-1} \leq \mathcal{H}^{N-1}(B_r(x_0) \cap \partial \{u>0\}) \leq \mathfrak{C}_0r^{N-1}
$$
for universal constants $\mathfrak{c}_0, \mathfrak{C}_0>$ and any ball $B_r(x_0)$ centred in free boundary points.

\section{Preliminaries}\label{Section2}

\hspace{0.6cm}\textbf{Notations.} Let us start this section by introducing some notations which we shall use throughout this article.
\begin{itemize}
\item[\checkmark] $N$ denotes de dimension of Euclidean space $\R^N$.
\item[\checkmark] $u_{+} = \max\{0, u\}$.
\item[\checkmark] $\mathfrak{F}(u_0, \Omega) \defeq \partial \{u_0 >0\} \cap \Omega$ shall mean the free boundary.
\item[\checkmark] $\mathcal{L}^N$ denotes the $n$-dimensional Lebesgue measure.
\item[\checkmark] $\mathcal{H}^{N-1}$ denotes the $(N-1)$-dimensional Hausdorff measure.
\item[\checkmark] $\Omega^{\prime} \Subset \Omega$ means that $\Omega^\prime\subset\overline{\Omega^\prime}\subset\Omega$, and $\overline{\Omega^\prime}$ is compact ($\Omega^\prime$ is compactly contained in $\Omega$).
\item[\checkmark] $B_r(x_0)$ denotes the ball of center $x_0$ and radius $r$. When the center is $0$ we just write $B_r$.
\item[\checkmark] $$\displaystyle \mathcal{S}_r[u](x_0) \defeq   \sup_{B_r(x_0)} u(x), \qquad  \displaystyle \mathcal{I}_r[u](x_0) \defeq   \inf_{B_r(x_0)} u(x).
$$
Moreover, we will omit the center of the ball when $x_0 = 0$.

\end{itemize}

\begin{definition}[{\bf Weak solution}] We say that $u \in L^1_{\text{loc}}(\Omega)$ is a weak supersolution (resp. subsolution) to
\begin{equation}\label{eqweaksol}
     -\div(\Phi(x, u, \nabla u)) + \Psi(x, u, \nabla u) = 0 \quad \text{in }  \Omega,
\end{equation}
if it is weakly differentiable in $\Omega$, i.e., all its weak derivatives of first order exist, $\Phi(\cdot, u, \nabla u) \in L^1_{\text{loc}}(\Omega)$ and for all $0\leq \varphi \in C^1_0(\Omega)$ it holds that
$$
\displaystyle \int_{\Omega} \Phi(x, u, \nabla u)\cdot \nabla \varphi(x)\,dx \geq \int_{\Omega} \Psi(x, u, \nabla u)\varphi(x)\,dx \quad \left(\text{resp.}\,\,  \leq \int_{\Omega} \Psi(x, u, \nabla u)\varphi(x)\,dx\right).
$$
Finally, we say that $u$ is a weak solution to \eqref{eqweaksol} when it is simultaneously a weak super-solution and a weak sub-solution.
\end{definition}

 We will also use the following useful comparison result for weak solutions (see for instance \cite{Diaz}).

\begin{lemma}[{\bf Comparison Principle}]\label{comparison} Let $\Omega \subset \R^N$ be a bounded open set, $\lambda_0 \in L^{\infty}(\Omega)$ and $f \in C([0, \infty))$ a non-negative and non-decreasing function. Assume that we have that
$$
 -\div(\Phi(x, v, \nabla v)) + \lambda_0(x)f(v)\chi_{\{v>0\}} \leq 0 \leq -\div(\Phi(x, u, \nabla u)) + \lambda_0(x)f(u)\chi_{\{u>0\}} \quad \mbox{in} \quad \Omega.
$$
in the weak sense with $\Phi$ fulfilling (H1)--(H4). If $v \leq u $ in $\partial \Omega$ then $v \leq u$ in $\Omega$.
\end{lemma}

An consequence from Lemma \ref{comparison} is the following existence of dead core solutions.

\begin{theorem}[{\bf Existence/uniqueness of dead core solutions}]\label{ExisThm} Suppose that the assumptions of Lemma \ref{comparison} are satisfied. If $u_{\ast}$ is a sub-solution and $u^{\ast}$ is a super-solution to \eqref{DCP1} with $g \in C^0(\partial \Omega)$ not identically zero such that $u_{\ast} = g = u^{\ast}$ on $\partial \Omega$, then there exists a non-negative function $u$ fulfilling \eqref{DCP1} in the weak sense with $u_{\ast} \leq u \leq u^{\ast}$ in $\Omega$. Moreover, such a solution is unique.
\end{theorem}

The following Serrin's Harnack inequality will be useful for our arguments. For the case $p<N$, this inequality follows by specializing \cite[Theorem 5]{Ser64}, with $\ve=1$, $R=\frac{1}{2}$, $\alpha=p$, $e=g=c=d=0$ and $\displaystyle f=\lam_+\sup_{B_{1/2}} u^q$.  The case $p\geq N$ is obtained by using \cite[Theorem 6 and Theorem 9]{Ser63} with $R=\frac{1}{2}$, $\alpha=p$, $e=g=0$ and $\displaystyle f=\lam^+ \sup_{B_{1/2}} u^q$.

\begin{theorem}[{\bf Harnack inequality}] \label{harnack}
Let $u$ be a non-negative weak solution to \eqref{Maineq} in the ball
	 $B_\frac{3}{2}$ with $\Phi$ fulfilling (H1)--(H4). Then
	\begin{equation*}
	  \displaystyle \mathcal{S}_{\frac{1}{2}}[u] \leq C(N, p)\left( \mathcal{I}_{\frac{1}{2}}[u]+\lam_+^{\frac{1}{p-1}}\mathcal{S}_1[u] ^{\frac{q}{p-1}}\right).
	\end{equation*}
\end{theorem}

\section{Sharp regularity estimates along free boundary}\label{Sec3}

In this section we will obtain sharp regularity estimates for weak solutions to dead core type problems. Such results come out by combining the Serrin's Harnack inequality with the scaling invariance of the equation and the optimal scaling for solutions at free boundary points.

\vspace{0.3cm}

{\bf Proof of Theorem \ref{MThm}}: Notice that the scaled and normalized function
$$
  v_r(x) \defeq \frac{u(x_0 + rx)}{r^{\frac{p}{p-1-q}}},
$$
satisfies the equation
$$
    -\div(\Phi_r(x, v_r, \nabla v_r)) + \tilde\lambda_0(x)(v_r)_{+}^q  = 0\quad \text{in} \quad B_1,
$$
in the weak sense, where
$$
  \Phi_r(x, s,\xi) \defeq r^{-\frac{(1+q)(p-1)}{p-1-q}}\Phi\left(x_0+rx,r^\frac{p}{p-1-q}s,r^{\frac{1+q}{p-1-q}}\xi\right) \quad \text{and} \quad \tilde\lambda_0(x) \defeq \lambda_0(x_0 + rx).
$$

Observe that $\Phi_r$ satisfies the same structural conditions as $\Phi$.  Indeed, (H1) and (H2) are trivially satisfied. Condition (H3) on $\Phi$ implies that it is also fulfilled by $\Phi_r$
\begin{align*}
	 \langle \Phi_r(x,s,\xi),\xi\rangle  &= \left\langle r^{-\frac{(1+q)(p-1)}{p-1-q}} \Phi\left(x_0+rx,r^\frac{p}{p-1-q}s,r^\frac{1+q}{p-1-q}\xi\right),\xi\right\rangle\\
	&=r^{-\frac{(1+q)p}{p-1-q}} \left\langle  \Phi\left(x_0+rx,r^\frac{p}{p-1-q}z,r^\frac{1+q}{p-1-q}\xi\right),r^\frac{1+q}{p-1-q} \xi\right\rangle\\	
	&\geq c_1 |\xi|^p  - c_2 |z|^p r^p\\	
	&\geq c_1 |\xi|^p - c_2 |z|^p
\end{align*}
since $0\leq q < p-1$ and $0<r<1$.

Similarly, condition (H4) on $\Phi$ implies that it is also fulfilled by $\Phi_r$,
\begin{align*}
	|\Phi_r(x,z,\xi)|&\leq r^{-\frac{(1+q)(p-1)}{p-1-q}} \left(c_3 |z|^{p-1}r^\frac{p(p-1)}{p-1-q} +c_4 |\xi|^{p-1} r^\frac{(p-1)(1+q)}{p-1-q}\right)\\
	&=c_3 |z|^{p-1} r^{p-1} + c_4 |\xi|^{p-1} \\
	&\leq c_3 |z|^{p-1}  + c_4 |\xi|^{p-1}
\end{align*}
since $p>1$ and $0<r<1$.

By applying   Theorem \ref{harnack} we obtain
\begin{equation}\label{HarnIneq}
  \displaystyle \mathcal{S}_\frac12 [v_r] \leq C(N, p)\left(\mathcal{I}_{\frac{1}{2}} [v_r] +\lam_+^{\frac{1}{p-1}}\mathcal{S}_1 [v_r]^{\frac{q}{p-1}}\right).
\end{equation}
For $q=0$, by scaling back \eqref{HarnIneq} in terms of $u$, the proof is immediate, see \cite{LeeShah}. Thus, consider $0<q<p-1$ and $\frac{\tau}{4} \leq c \leq \frac{\tau}{2}$. In order to iterate \eqref{HarnIneq} in relation to $r_k \defeq \frac{c}{2^k}$ for $0\leq k \leq n_0$ such that $r_{n_0} = r$ for some $n_0 \in \mathbb{N}$, we must consider two possibilities.

First, if it holds that
$$
    r_{k}^{\frac{-p}{p-1-q}} \mathcal{I}_{r_k} [u](x_0)   \leq  \lam_+^{\frac{1}{p-1}} \left( r_{k-1}^{\frac{-p}{p-1-q}} \mathcal{S}_{r_{k-1}} [u](x_0) \right)^{\frac{q}{p-1}} \quad \text{for} \quad 1\leq k\leq n_0,
  $$
we arrive at
 $$
  \displaystyle   r_{n_0}^{\frac{-p}{p-1-q}} \mathcal{S}_{r_{n_0}} [u](x_0) \leq C(N, p, q, \lam_+)\left( r_0^{\frac{-p}{p-1-q}} \mathcal{S}_{r_0} [u](x_0)\right)^\frac{q}{p-1} \leq \mathfrak{C}
 $$
 from where we deduce that $\displaystyle \mathcal{S}_{r}[u](x_0)\leq \mathfrak{C} r^{\frac{p}{p-1-q}}$.

 Conversely, if for some $k_0\leq n_0$ it holds that
  $$
	  r_{k_{0}}^{\frac{-p}{p-1-q}} \mathcal{I}_{r_{k_0}} [u](x_0) \geq  \lam_+^{\frac{1}{p-1}} \left( r_{k_0-1}^{\frac{-p}{p-1-q}} \mathcal{S}_{r_0-1} [u](x_0) \right)^{\frac{q}{p-1}}
  $$
  and
  $$
 r_{k}^{\frac{-p}{p-1-q}} \mathcal{I}_{r_k} [u](x_0)  \leq  \lam_+^{\frac{1}{p-1}} \left( r_{k-1}^{\frac{-p}{p-1-q}} \mathcal{S}_{r_{k-1}} [u](x_0) \right)^{\frac{q}{p-1}}  \quad \text{for} \quad k_0< k\leq N_0,
  $$
  then  we arrive at
  $$
      r_{n_0}^{\frac{-p}{p-1-q}} \mathcal{S}_{r_{n_0}} [u](x_0) \leq \mathfrak{c} r_{k_0}^{\frac{-p}{p-1-q}}  \mathcal{I}_{r_{k_0}} [u](x_0) \leq \mathfrak{c} r_{n_0}^{\frac{-p}{p-1-q}} \mathcal{I}_{r_{n_0}} [u](x_0),
  $$
  where $\mathfrak{c}=\mathfrak{c}(N, p, q, \lam_+)$, from where the proof of the result follows.\qed

\vspace{0.3cm}

 As an immediate consequence of Theorem \ref{MThm} we obtain a finer decay near free boundary points. Precisely, a dead core solution $u$ arrives at its null set as the distance to the free boundary.

\begin{corollary}\label{CorControldist}
Let $u$ be a weak solution to \eqref{Maineq} and $x_0 \in \{u>0\} \cap \Omega$. Then, for a universal constant $\mathfrak{C}>0$ there holds that
$$
  u(x_0) \leq \mathfrak{C}.\dist(x_0, \partial \{u>0\})^{\frac{p}{p-1-q}}.
$$
\end{corollary}

\begin{remark}
 Following the same arguments that in the proof of Theorem \ref{MThm}, it is possible to obtain similar regularity estimates for a family of problems with a general (non $q$-homogeneous) non-linear absorption term $\mathfrak{f}: [0, \|u\|_{\infty}] \to \R_{+}$, i.e.,
$$
    -\div(\Phi(x, u, \nabla u)) + \mathfrak{f}(u)  = 0\quad \mbox{in} \quad \Omega,
$$
provided $\mathfrak{f}(0)=0$ and for some constant $\mathfrak{C}>0$ it holds that
$$
    \mathfrak{f}(s) \leq \mathfrak{C}s^q,
$$
for all $0< s\ll 1$. Some interesting examples include
$$
\mathfrak{f}(u) = \left\{
\begin{array}{lcl}
  \lambda_0(x)(e^{u_{+}^t}-1) & \text{for}& t \geq q>0 \\
  \lambda_0(x)\ln(u_{+}^t+1) &\text{for}& t \geq  q>0 \\
  \lambda_0(x)u_{+}^q\ln(u^t+1) & \text{for}& t > 0\\
  \lambda_0(x)\frac{u_{+}^q}{(1+u^t)^m} & \text{for}& t> 0 \,\,\,\text{and}\,\,\,0<m\leq q.
\end{array}
\right.
$$

We have chosen the case $\mathfrak{f}(u) = \lambda_0(x)u_{+}^q(x)$ in order to introduce the main ideas of  Theorem  \ref{MThm}.
\end{remark}

\medskip

In this final part we will analyse the critical case $q=p-1$. For this purpose, let us define the operator,
\begin{equation}\label{localeq3}
    \mathcal{Q}_{\Phi}[u](x)  \defeq -\div(\Phi(x, u(x), \nabla u(x))) + \lambda_0(x)u^{p-1}(x).
\end{equation}

Notice that this operator is critical since all the estimates established in the previous sections deteriorate as $q$ approaches $p-1$. Particularly, it follows from Theorem \ref{MThm} that if $u$ vanishes at an interior point $x_0 \in \Omega$, then $D^{k} u (x_0) = 0$  for all $k\in \mathbb{N}$, i.e., any vanishing interior point is an infinite order zero (compare it with the Unique continuation property). By means of a geometric barrier argument (Hopf's boundary type reasoning), which explore the scaling invariance of the operator $\mathcal{Q}_{\Phi}$, we shall prove that a non-negative solution to \eqref{localeq3} cannot vanish at interior points, unless they are identically zero.

\begin{theorem}[{\bf Strong Maximum Principle}] \label{strongmp} Let $u$ be a non-negative weak solution to \eqref{localeq3}. If there exists a point $x_0 \in \Omega$ such that $u(x_0)=0$, then $u \equiv 0$ in $\Omega$.
\end{theorem}

\begin{proof}
The prove will follow by \textit{reductio ad absurdum}. For this end, let $x_0 \in \Omega$ such that $u(x_0) > 0$ and suppose without loss of generality that
$$
   d_0 \defeq \dist(x_0,\partial\{u>0\}) < \frac{1}{5}\dist(x_0,\partial \Omega).
$$
We have that $u$ is locally bounded due to Comparison Principle. Now, for fixed values of $\mathcal{A}>0$ and $s>0$ (large enough) we define the following barrier function:
$$
   \Theta_{\mathcal{A}, s}(x) \defeq \mathcal{A}\frac{e^{-s \frac{|x-x_0|^2}{d_0^2}}-e^{-s}}{e^{-\frac{s}{4}}-e^{-s}},
$$
for which, a straightforward calculation shows that
\begin{equation}\nonumber
\left\{
\begin{array}{rclcl}
\mathcal{Q}_{\Phi}[\Theta_{\mathcal{A}, s}](x) & \leq & 0& \mbox{in} & B_{d_0}(x_0) \setminus B_{\frac{d_0}{2}}(x_0) \\
\Theta_{\mathcal{A}, s} & = & \mathcal{A} & \mbox{in} &  \partial B_{\frac{d_0}{2}}(x_0)\\
 \Theta_{\mathcal{A}, s} & = & 0 & \mbox{in} & \partial B_{d_0}(x_0).
\end{array}
\right.
\end{equation}
Notice that

\begin{equation}\label{gradTheta}
\inf\limits_{B_{d_0(x_0)} \setminus B_{\frac{d_0}{2}}(x_0)}|\nabla \, \Theta_{\mathcal{A}, s} (x) | \geq \frac{\mathcal{A} s d_0^{-1}  e^{-s}}{e^{-\frac{s}{4}}-e^{-s}} \geq \frac{5\mathcal{A}s e^{-s}}{\dist(x_0,\partial \Omega)(e^{-\frac{s}{4}}-e^{-s})}\defeq \kappa >0.
\end{equation}

On the other hand, for any constant $\zeta>0$,  the barrier $\zeta \Theta_{\mathcal{A}, s}$ still being a subsolution in $B_{d_0(x_0)} \setminus B_{\frac{d_0}{2}}(x_0)$. In consequence,
$$
   \mathcal{Q}_{\Phi}[\zeta  \Theta_{\mathcal{A}, s}](x) \leq 0 \leq \mathcal{Q}_{\Phi}[u](x) \quad \text{in} \quad B_{d_0(x_0)} \setminus B_{\frac{d_0}{2}}(x_0).
$$
Moreover, by taking $\zeta_0 \in(0,1)$ small enough such that
$$
  \displaystyle \zeta_0 \mathcal{A} \leq \mathcal{I}_{\frac{d_0}{2}} [u](x_0)
$$
we obtain
$$
    \zeta_0 \Theta_{\mathcal{A}, s} \leq u \quad \mbox{in} \quad \partial B_{d_0}(x_0) \cup \partial B_{\frac{d_0}{2}}(x_0).
$$
Thus, by using Comparison Principle (Lemma \ref{comparison}) we obtain that
\begin{equation}\label{cprin}
    \zeta_0\Theta_{\mathcal{A}, s} \leq  u \quad \mbox{in} \quad B_{d_0}(x_0) \setminus B_{\frac{d_0}{2}}(x_0).
\end{equation}
Now, for any $\max\{0, p-2\}<m<p-1$  \eqref{Maineq} can be rewritten as
$$
   -\div(\Phi(x, u, \nabla u)) + g(x) u_{+}^{m}  = 0 \quad \mbox{in} \quad \Omega,
$$
where $g(x)\defeq \lambda_0(x) u_{+}^{p-m-1}(x)$. Thus, for $z \in \partial B_{d_0} \cap \partial\{u>0\}$,   we can invoke Theorem \ref{MThm}   to obtain
\begin{equation}\label{eqest}
   \mathcal{S}_r [u](z) \leq \mathfrak{C}(N, p, m, \|g\|_{L^{\infty}(\Omega)}) r^{\frac{p}{p-m-1}}\leq \mathfrak{C} r^p
\end{equation}
for $r\ll1$ small enough, since $g$ fulfils the assumptions of such a theorem.

Observe that from \eqref{gradTheta} one can see that
\begin{equation*}
	\zeta_0 \kappa |x-z| \leq  \zeta_0 |\Theta_{\mathcal{A}, s}(x)-\Theta_{\mathcal{A}, s}(z)|= \zeta_0 \Theta_{\mathcal{A}, s}(x).
\end{equation*}
Moreover, from \eqref{cprin} we get
\begin{equation*}
	\zeta_0 \Theta_{\mathcal{A}, s}(x) \leq u(x) \le \mathcal{S}_{|x-z|} [u](z).
\end{equation*}
Mixing up the last two equations together with \eqref{eqest} we obtain that
\begin{equation}\label{eqcontr}
  \zeta_0\kappa  \leq \mathfrak{C} |x-z|^{p-1}
\end{equation}
for all $x \in (B_{d_0}(x_0) \setminus B_{d_0/2}(x_0)) \cap B_r(z)$ provided that $r$ is small enough.

Since the constants $\zeta_0$, $\kappa$ and $\mathfrak{C}$ are fixed and  independent on $x$, from \eqref{eqcontr} we arrive to a contradiction  since $x\in  (B_{d_0}(x_0) \setminus B_{{d_0/2}}(x_0)) \cap B_r(z)$  can be taken such that the difference $|x-z|$ is arbitrarily small. Such absurd proves the result.
\end{proof}

\begin{example}
According to Theorem \ref{strongmp}, when $q=p-1$, non-trivial weak solutions to \eqref{Maineq} must be strictly positive. For example, if $\lambda_0$ is a positive constant, fixed any direction $i = 1, \ldots, N$  we have that
$$
   u(x) = u(x_1, \cdots, x_i, \cdots, x_N) = e^{\sqrt[p]{\frac{\lambda_0}{p-1}}. x_i}
$$
is a strictly positive solution to
$$
     -\Delta_p u(x) + \lambda_0.u^{p-1}(x)  = 0\quad \mbox{in} \quad \Omega \subset \R^N,
$$
for any $p>1$.
\end{example}

\section{Non-degeneracy properties}\label{Sec4}

This Section is devoted to prove some geometrical and measure properties that play an essential role in the description of solutions to free boundary problems of dead core type. For this purpose, we will assume (from now on) the following properties on $\Phi$, which are, in fact are stronger than (H1)  and (H3), namely,

\begin{itemize}
	\item[(A1)]  $\Phi \in C^1(\Omega\times\R\times \R^N\setminus \{0\}; \R^N)$
	
	\item[(A2)] There exists a constant $\kappa_1>0$ such that for $p\geq 2$
	
	\begin{equation} \label{eq1Phi}
	   \displaystyle \sum_{i, j = 1}^{N} \left|\frac{\partial \Phi_i}{\partial \xi_j}(x, z, \xi)\right| \leq \kappa_1|\xi|^{p-2},
	\end{equation}

	\begin{equation} \label{eq2Phi}
	   \displaystyle \sum_{i, j = 1}^{N} \left|\frac{\partial \Phi_i}{\partial x_j}(x, z, \xi)\right| + \sum_{i, j = 1}^{N} \left|\frac{\partial \Phi_i}{\partial z}(x, z, \xi)\right|\leq \kappa_1|\xi|^{p-1}
	\end{equation}
	for a.e. $x \in \Omega$, all $z \in \R$ and $\xi \in \R^N \setminus \{0\}$.
\end{itemize}

The next result gives exactly the growth rate at which non-negative weak solutions leave their dead core sets. More precisely, the theorem establishes a $C^{\frac{p}{p-1-q}}$ growth estimate from below, which together with Corollary \ref{CorControldist} implies that $u$ leaves the dead core set trapped by the graph of two functions of the order $\dist(x, \partial\{u>0\})^{\frac{p}{p-1-q}}$. It is worth to mention  that we will not require the $p$-ellipticity hypothesis. However, we are assuming the monotonicity hypothesis from Section \ref{Sec3} in order to use the comparison principle.

\begin{theorem}[{\bf Non-degeneracy}]\label{LGR} Let $u$ be a non-negative, bounded weak solution to \eqref{Maineq} in $\Omega$ and let $\Omega^{\prime} \Subset \Omega$, $x_0 \in \overline{\{u >0\}} \cap \Omega^{\prime}$ be a generic point in the closure of the non-coincidence set. Assume also that $p>2 + q$ and $\displaystyle \lambda_{-} = \inf_{\Omega} \lambda_0(x)>0$. Then for all $0<r<\min\{1, \dist(\Omega^{\prime}, \partial \Omega)\}$ there holds
\begin{equation}\label{nondeg est}
   \displaystyle \sup_{\partial B_r(x_0)} \,u(x) \geq \mathfrak{c}_0(p, q, \kappa_1, \lambda_{-}) r^{\frac{p}{p-1-q}}.
\end{equation}
\end{theorem}

\begin{proof}
Notice that, due to the continuity of solutions, it is sufficient to prove that such   estimate is satisfied just at points within $\{u>0\} \cap \Omega$.

First of all,  let us define the scaled function
$$
    u_r(x) \defeq \frac{u(x_0+rx)}{r^{\frac{p}{p-1-q}}}.
$$
Now, let us introduce the auxiliary barrier function
$$
    \displaystyle \Psi (x) \defeq \mathfrak{C} |x|^{\frac{p}{p-1-q}}
$$
for a positive constant $\mathfrak{C}$ to determinate \textit{a posteriori} and the re-scaled vector field

$$
\Phi_r(x, z,\xi) \defeq r^{-\frac{(1+q)(p-1)}{p-1-q}}\Phi\left(x_0+rx,r^\frac{p}{p-1-q}z,r^{\frac{1+q}{p-1-q}}\xi\right).
$$
Observe that the re-scaled function $\Phi_r$ also satisfies hypothesis $(A1)$ and $(A2)$ (The monotonicity hypothesis is satisfied trivially). Indeed, since $0<r<1$ from \eqref{eq1Phi}

$$
\begin{array}{rcl}
   \displaystyle \sum_{i, j = 1}^{N} \left|\frac{\partial (\Phi_i)_r}{\partial \xi_j}(x, z, \xi)\right| & = & \displaystyle r^{-\frac{(1+q)(p-1)}{p-1-q}} \sum_{i, j = 1}^{N} \left|\frac{\partial \Phi_i}{\partial \xi_j}(x_0+rx,r^\frac{p}{p-1-q}z,r^{\frac{1+q}{p-1-q}}\xi)\right|\\
   & \leq  &  \kappa_1 r^{-\frac{(1+q)(p-1)+1+q}{p-1-q}}\left|r^{\frac{1+q}{p-1-q}}\xi\right|^{p-2}\\
   & = & \kappa_1.
\end{array}
$$
Now, from \eqref{eq2Phi} we obtain that
$$
\begin{array}{rcl}
 \displaystyle \sum_{i, j = 1}^{N} \left|\frac{\partial (\Phi_i)_r}{\partial x_j}(x, z, \xi)\right| + \sum_{i, j = 1}^{N} \left|\frac{\partial (\Phi_i)_r}{\partial z}(x, z, \xi)\right|&= & r^{-\frac{(1+q)(p-1)}{p-1-q}}\displaystyle \sum_{i, j = 1}^{N} \left|\frac{\partial \Phi_i}{\partial x_j}(x_0+rx,r^\frac{p}{p-1-q}z,r^{\frac{1+q}{p-1-q}}\xi)\right|\\
   & + & r^{-\frac{(1+q)(p-1)}{p-1-q}}\displaystyle \sum_{i, j = 1}^{N} \left|\frac{\partial \Phi_i}{\partial z}(x_0+rx,r^\frac{p}{p-1-q}z,r^{\frac{1+q}{p-1-q}}\xi)\right|\\
 &\leq & r^{-\frac{(1+q)(p-1)}{p-1-q}+1} \kappa_1 \left|r^{\frac{1+q}{p-1-q}} \xi\right|^{p-1}\\
 &=&r \kappa_1 | \xi|^{p-1}\\
 &\leq &\kappa_1 | \xi|^{p-1}.
\end{array}
$$
A straightforward computation  using \eqref{eq1Phi} and \eqref{eq2Phi} shows that
$$
     -\div(\Phi_r(x, \Psi, \nabla \Psi)) + \hat{\lambda}_0\left( x  \right).\Psi^{q}(x) \geq 0 \geq -\div(\Phi_r(x, u_r, \nabla u_r)) + \hat{\lambda}_0\left( x  \right). (u_r)_{+}^{q}(x) \quad \text{in} \quad B_1,
$$
provided that $\mathfrak{C}$ fulfils the condition
\begin{equation} \label{condi}
	\kappa_1 \mathfrak{C}^{p-1-q} \left[ \left(\frac{p}{p-1-q}\right)^{p-1}+ \mathfrak{C}^{-1}\left(\frac{p}{p-1-q}\right)^{p-2}+1 \right]  \leq \lam_-,
\end{equation}
where above we have defined $\hat \lambda_0(x) \defeq \lambda_0(x_0 + rx)$. Observe that such a condition \eqref{condi} is naturally possible, because the continuous function
$$
  \mathfrak{h}(t) \defeq \mathfrak{a}t^{\alpha +1} + \mathfrak{b}t^{\alpha} +\mathfrak{c},
$$
where
$$
\alpha \defeq p-2-q> 0, \,\,\, \mathfrak{a} \defeq \kappa_1\left[\left(\frac{p}{p-1-q}\right)^{p-1}+1\right],\,\,\, \mathfrak{b} \defeq \kappa_1\left(\frac{p}{p-1-q}\right)^{p-2} \,\,\, \text{and} \,\,\, \mathfrak{c} \defeq -\lambda_{-}
$$
admits at least a real root according to the Intermediate value theorem.

Finally, if $u_r \leq \Psi$ on the whole boundary of $B_1$, then, the Comparison Principle would imply that
$$
   u_r \leq \Psi \quad \mbox{in} \quad B_1,
$$
which clearly contradicts the assumption that $u_r(0)>0$. Therefore, there exists a point $Y \in \partial B_1$ such that
$$
      u_r(Y) > \Psi(Y) = \mathfrak{C},
$$
and scaling back we finish the proof of the theorem.
\end{proof}

\begin{remark} Let us stress that if we consider $\xi \mapsto \Phi(\xi)$, i.e., the vector field $\Phi$ does not depend on lower order terms, then we can remove the assumption $p>2+q$  just for $p \geq 2$. Moreover, in the case of $p$-Laplace operator we find explicitly that
$$
    \mathfrak{C} = \left[\frac{\lambda_{-}}{\kappa_1} \frac{ \left(p-1-q \right)^{p} }{ p^{p-1}(pq+N(p-1-q) )}\right]^{\frac{1}{p-1-q}}.
$$
\end{remark}

An interesting piece of information is that as consequence of  Theorem \ref{LGR} we obtain the following finer growth for point near the free boundary: given $x_0 \in \{u>0\} \cap \Omega$, there holds that
$$
  u(x_0) \geq \mathfrak{C}\dist(x_0, \partial \{u>0\})^{\frac{p}{p-1-q}}.
$$

\begin{corollary}\label{CorNonDeg}
Let $u$ be a non-negative, bounded weak solution to \eqref{DCP1} in $\Omega$ and $\Omega^{\prime} \Subset \Omega$. Given $x_0 \in \{u>0\} \cap \Omega^{\prime}$, there exists a universal constant $\mathfrak{C}_{\sharp}>0$ such that
$$
  u(x_0) \geq \mathfrak{C}_{\sharp}\dist(x_0, \partial \{u>0\})^{\frac{p}{p-1-q}}.
$$
\end{corollary}
\begin{proof}
  Suppose that $\mathfrak{C}_{\sharp}$ does not exist. Then there exist a sequence $x_k \in \{u>0\} \cap \Omega^{\prime}$ with
$$
d_k \defeq \dist(x_k, \partial \{u>0\} \cap \Omega^{\prime}) \to 0 \quad \text{as} \quad k \to \infty \quad \text{and} \quad u(x_k) \leq  k^{-1} d_k^{\frac{p}{p-1-q}}.
$$
Now, let us define the auxiliary function $v_k:B_1 \to \R$ by
$$
   v_k(y) \defeq \frac{u(x_k+d_ky)}{d_k^{\frac{p}{p-1-q}}}
$$
and the vector field
$$
\Phi_k(x, z, \xi) \defeq d_k^{-\frac{(1+q)(p-1)}{p-1-q}}\Phi\left(x_k+d_kx,d_k^\frac{p}{p-1-q}z,d_k^{\frac{1+q}{p-1-q}}\xi\right).
$$
It is easy to check that
\begin{enumerate}
  \item $v_k \geq 0$ in $B_1$.
  \item $-\div(\Phi(x, v_k, \nabla v_k) + \lambda_0(x_k+d_ky)v_k^q = 0$ in $B_1$ in the weak sense.
  \item $v_k(y) \leq \mathfrak{C}(N, p).d_k^{\alpha} + \frac{1}{k} \,\, \forall\,\, y \in B_1$ according to local H\"{o}lder regularity of weak solutions.
\end{enumerate}
From the Non-degeneracy Theorem \ref{LGR} and the last sentence we obtain that
\begin{equation}\label{estim.44}
  \displaystyle 0<\mathfrak{C}_0\cdot \left(\frac{1}{2}\right)^{\frac{p}{p-1-q}} \leq \sup_{B_{\frac{1}{2}}} v_k(y) \leq \max\{1, \mathfrak{C}(N, p)\}\cdot \left(d_k^{\alpha}+\frac{1}{k}\right) \to 0 \quad \text{as} \quad k \to \infty,
\end{equation}
which clearly yields a contradiction. This concludes the proof.
\end{proof}

\section{Some consequences}

In the following, we will present some consequences arising from the growth rates and the non-degeneracy property for quasi-linear dead core problems.

As a first consequence of Theorem \ref{MThm}, we improve the local growth estimates for first derivatives of dead core solutions. In fact, we obtain a finer gradient control near the free boundary.

\begin{corollary}\label{GradEst} Let $u$ be a bounded weak solution to \eqref{Maineq} in $B_1$. Then, for any point $z \in   \{u > 0\} \cap B_{\frac{1}{2}}$, there holds
\begin{equation}
  |\nabla u(z)| \leq \mathfrak{C}\dist(z, \partial \{u>0\})^{\frac{1+q}{p-1-q}}.
\end{equation}
\end{corollary}
\begin{proof}
First of all, fix $z\in   \{u > 0\} \cap B_{\frac{1}{2}}$ and denote $r\defeq  \dist(z, \partial \{u>0\}$. Now, select $x_0 \in \partial \{u>0\}$ a free boundary point which achieves the distance, i.e.,
$$
  r = |z-x_0|.
$$
According to Theorem \ref{MThm} we have that
\begin{equation}\label{eq4.4}
  \displaystyle \sup_{B_{r}(z)} u(x) \leq \sup_{B_{2r}(x_0)} u(x) \leq \mathfrak{C} r^{\frac{p}{p-1-q}}.
\end{equation}
Next, we define the scaled auxiliary function $\omega: B_1 \to \R_{+}$ by
$$
   \omega(x) \defeq \frac{u(z+rx)}{r^\frac{p}{p-1-q}}.
$$
As previously, $\omega$ fulfils the following equation
\begin{equation}\label{eq4.5}
  \div(\Phi_{r}(x, \omega, \nabla \omega)) = \hat{\lambda}_{0}(x)\omega^q(x) \quad \text{in} \quad B_1
\end{equation}
in the weak sense, where
$$
  \Phi_{r}(x, s,\xi) \defeq r^{-\frac{(1+q)(p-1)}{p-1-q}}\Phi\left(z+rx,r^\frac{p}{p-1-q}s,r^{\frac{1+q}{p-1-q}}\xi\right) \quad \text{and} \quad \hat{\lambda}_0(x) \defeq \lambda_0(z + r x).
$$
It is a straightforward to verify that $\Phi_{r}$ also  satisfies  hypothesis (H1)--(H4). Moreover, from \eqref{eq4.4} we get that
$$
   \displaystyle \sup_{B_1} \omega(x) \leq \mathfrak{C}.
$$
Finally, by invoking the uniform gradient estimates from \cite{Choe1}, \cite{DB0}, \cite[Chapter 4]{LU} and \cite[Proposition 2]{Tolk} for bounded solutions we obtain that
$$
    \displaystyle  \frac{1}{r^{\frac{1+q}{p-1-q}}} |\nabla u(z)| = |\nabla \omega(0)| \leq \mathfrak{C}_0.
$$
This finishes the proof of the corollary.
\end{proof}

The Non-degeneracy estimate from Corollary \ref{CorNonDeg} implies particularly a non-degeneracy property in measure. As it was commented in the introduction, this estimate is useful in several qualitative contexts of the theory of free boundary problems.

\begin{theorem}[{\bf Non-degeneracy in measure}] Let $u$ be a weak solution to \eqref{Maineq}. Given $\Omega^{\prime} \subset \Omega$ there exist $\rho_0>0$ and $\kappa>0$ depending only on $\Omega^{\prime}$ and universal parameters such that
$$
  \mathcal{L}^N\left(\Omega^{\prime} \cap \{0<u(x)<\rho^{\frac{p}{p-1-q}}\}\right)\leq \kappa\rho
$$
for any $\rho\leq \rho_0$.
\end{theorem}
\begin{proof} Fix $r_0>0$ (to be chosen \textit{a posteriori}) and let $\rho_0$ given by
$$
   \rho_0 = \min\{u(x) \suchthat x \in \Omega^{\prime} \,\,\,\text{and}\,\,\, \dist(x, \partial \{u>0\})\geq r_0\}.
$$
Now, let $x_0 \in \Omega^{\prime}$ such that $\dist(x_0, \partial \{u>0\})< r_0$. According to Corollary \ref{CorNonDeg}
$$
   u(x_0) \geq \mathfrak{C}\dist(x_0, \partial \{u>0\})^{\frac{p}{p-1-q}}.
$$
For this very reason, if $r_0> \dist(x_0, \partial \{u>0\}) \geq \frac{\rho}{\mathfrak{C}^{\frac{p-1-q}{p}}}$,  we obtain that $u(x_0) \geq \rho^{\frac{p}{p-1-q}}$. Therefore,
$$
  \mathcal{L}^N\left(\Omega^{\prime} \cap \left\{0<u(x)<\rho^{\frac{p}{p-1-q}}\right\}\right)\leq \mathcal{L}^N\left(\left\{\Omega^{\prime}\cap \{u>0\} \suchthat \dist(x, \partial \{u>0\}) < \frac{\rho}{\mathfrak{C}^{\frac{p-1-q}{p}}}\right\}\right) \leq \kappa\rho
$$
for some universal $\kappa>0$.
\end{proof}

\medskip

As soon as we prove the sharp asymptotic behaviour for our free boundary problem, it is a natural consequence to obtain some weak geometric properties of the phase zone. For this reason, we finish up this part by establishing that the positiveness region has uniform positive density along the free boundary. Particularly, the development of cusps along free boundary is inhibited.

 \begin{corollary}[{\bf Uniform positive density}]\label{UPDFB}Let $u$ be a non-negative weak solution to \eqref{Maineq} in $B_1$ and $x_0 \in \partial \{u > 0\} \cap B_{\frac{1}{2}}$ be a free boundary point. Then for any $0<\rho< \frac{1}{2}$,
$$
     \mathcal{L}^N(B_{\rho}(x_0) \cap\{u>0\})\geq \theta \rho^N,
$$
for a constant $\theta>0$ that depends only on  $p$ and $q$.
\end{corollary}
\begin{proof} Applying Theorem \ref{LGR} there exists a point $\hat{y} \in  \partial B_r(x_0) \cap \{u>0\}$ such that,
\begin{equation}\label{dens}
    u(\hat{y})\geq \mathfrak{c}_0 r^{\frac{p}{p-1-q}}.
\end{equation}
Moreover, from Theorem \ref{MThm} there exists   $\kappa>0$ small enough with universal dependence
such that
\begin{equation}\label{inclusion}
    B_{\kappa r}(\hat{y}) \subset  \{u>0\},
\end{equation}
where the constant $\kappa$ is given by
$$
   \kappa \defeq \left(\frac{\mathfrak{c}_0}{7\mathfrak{C}_0}\right)^{\frac{p-1-q}{p}}.
$$
In effect, if this were not true, it would exist a free boundary point $\hat{z} \in B_{\kappa r}(\hat{y})$. Consequently, from \eqref{dens}   we obtain that
$$
   \mathfrak{c}_0 r^{\frac{p}{p-1-q}} \leq u(\hat{y}) \leq \sup_{B_{\kappa r}(\hat{z})} u(x) \leq \mathfrak{C}_0(\kappa r)^{\frac{p}{p-1-q}} = \frac{1}{7}\mathfrak{c}_0 r^{\frac{p}{p-1-q}},
$$
which yields a contradiction. Therefore,
$$
    B_{\kappa r}(\hat{y}) \cap B_r(x_0) \subset  B_r(x_0) \cap \{u>0\}.
$$
Hence,
$$
     \mathcal{L}^N(B_{\rho}(x_0) \cap\{u>0\})\geq \mathcal{L}^N(B_{\rho}(x_0) \cap B_{\kappa r}(\hat{y}))\geq \theta r^N,
$$
which proves the result.
\end{proof}

\begin{definition}[{\bf $\zeta$-Porous set}] A set $S \in \R^N$ is said to be porous with porosity constant $0<\zeta \leq 1$ if there exists an $R > 0$ such that for each $x \in S$ and $0 < r < R$ there exists a point $y$ such that $B_{\zeta r}(y) \subset B_r(x) \setminus S$.
\end{definition}

\begin{corollary}[{\bf Porosity of the free boundary}]\label{CorPor} There exists a constant $0<\xi =  \xi(N, \lam, p, q) \leq 1$ such that
\begin{equation}\label{eqPor}
       \mathcal{H}^{N-\xi}\left(\partial \{u>0\}\cap B_{\frac{1}{2}}\right)< \infty.
\end{equation}
\end{corollary}

\begin{proof}
Let $R>0$ and $x_0\in\Omega$ be such that $\overline{B_{4R}(x_0)}\subset \Omega$. We will prove that $\partial \{u >0\} \cap B_R(x_0)$ is a $\frac{\delta}{2}$-porous set for a universal constant $0< \delta \leq 1$. For this purpose, let $x\in \partial \{u >0\} \cap B_{R}(x_0)$. For each $r\in(0, R)$ we have $\overline{B_r(x)}\subset B_{2R}(x_0)\subset\Omega$. Now, let $y\in\partial B_r(x)$ such that
$$
   u(y) = \sup\limits_{\partial B_r(x)} u(t).
$$
By Theorem \ref{LGR}
\begin{equation}\label{5.1}
    u(y)\geq \mathfrak{c}_0 r^{\frac{p}{p-1-q}}.
\end{equation}
On the other hand, near the free boundary, from Theorem \ref{MThm}
\begin{equation}\label{5.2}
    u(y)\leq \mathfrak{C}_0 d(y)^{\frac{p}{p-1-q}},
\end{equation}
where $d(y) \defeq \text{dist}(y, \partial \{u>0\} \cap \overline{B_{2R}(x_0)})$. From \eqref{5.1} and \eqref{5.2} we get
\begin{equation}\label{5.3}
    d(y)\geq\delta r
\end{equation}
for a positive constant $0<\delta \defeq \left(\frac{\mathfrak{c}_0}{\mathfrak{C}_0}\right)^{\frac{p-1-q}{p}}\leq1$.

Now, let $\hat{y}$ be in the segment joining  $x$ and $y$ be such that $|y-\hat{y}|=\frac{\delta r}{2}$, then there holds
\begin{equation}\label{5.4}
   B_{\frac{\delta}{2}r}(\hat{y})\subset B_{\delta r}(y)\cap B_r(x),
\end{equation}
indeed, for each $z\in B_{\frac{\delta}{2}r}(\hat{y})$
\begin{align*}
   |z-y|&\leq |z-\hat{y}|+|y-\hat{y}|<\frac{\delta r}{2}+\frac{\delta r}{2}=\delta r,\\
   |z-x|&\leq|z-\hat{y}|+\big(|x-y|-|\hat{y}-y|\big)\leq\frac{\delta r}{2}+\left(r-\frac{\delta r}{2}\right)=r.
\end{align*}
Finally, since by \eqref{5.3} $B_{\delta r}(y)\subset B_{d(y)}(y)\subset\{u>0\}$, we get
$$
   B_{\delta r}(y)\cap B_r(x)\subset\{u>0\},
$$
which together with \eqref{5.4} implies
$$
   B_{\frac{\delta}{2}r}(\hat{y})\subset B_{\delta r}(y)\cap B_r(x)\subset B_r(x)\setminus\partial\{u>0\}\subset B_r(x)\setminus \partial\{u>0\} \cap B_{R}(x_0).
$$
Therefore, $\partial\{u>0\} \cap B_{R}(x_0)$ is a $\frac{\delta}{2}$-porous set. Finally, the $(N-\xi)$-Hausdorff measure estimates in \eqref{eqPor} follows from \cite{KR}.
\end{proof}

Particularly, we conclude from Corollary \ref{CorPor} that  $\partial \{u>0\}$ has Lebesgue measure zero.

\vspace{0.3cm}

In this last part we shall establish the stability of the dead core (or coincidence) set. Precisely, we will prove that the $L^{\infty}$-norm of two dead core solutions controls the difference, in the measure theoretic sense, of their dead core sets.

Let us introduce the following notation
$$
   \mathfrak{K}[u] \defeq \{u=0\} \cap B_{\frac{1}{2}} \quad \mbox{and} \quad \mathfrak{K}[u]\vartriangle \mathfrak{K}[v] \defeq (\mathfrak{K}[u]\setminus \mathfrak{K}[v]) \cup (\mathfrak{K}[v]\setminus \mathfrak{K}[u]).
$$

\begin{theorem}\label{SDCThm} Let $v_1, v_2$  be two bounded weak solutions to \eqref{Maineq} under the hypothesis of Theorems \eqref{Maineq} and \ref{LGR} and fulfilling, for $0< \sigma \leq 1$
\begin{equation}\label{eqStabest}
   \|v_1-v_2\|_{L^{\infty}(B_1)} \leq \sigma^{\frac{p}{p-1-q}}.
\end{equation}
Then, for a constant $c = c(N, p, q)>0$ large enough there holds that
$$
   \mathcal{L}^N(\mathfrak{K}[v_1]\vartriangle \mathfrak{K}[v_2]) \leq c \sigma \quad \mbox{and} \quad \mathfrak{K}_{-c \sigma}[v_2] \subset \mathfrak{K}[v_1] \subset \left\{v_2 < \sigma^{\frac{\gamma+2}{\gamma+1-\mu}}\right\},
$$
where
$$
   \mathfrak{K}_{-c\sigma}[v_2] \defeq \{x \in \mathfrak{K}[v_2]: \dist(x, \{v_2>0\})> \sigma\}.
$$
\end{theorem}

\begin{proof}
Note that for any solution $u$ to \eqref{Maineq} it holds from Theorem \ref{MThm} and Corollary \ref{GradEst} that
\begin{equation}\label{eqinclus}
   \left\{x : \,\,0<u(x)<\hat{\mathfrak{C}} \sigma^{\frac{p}{p-1-q}}\right\} \cap B_r(x_0) \subset \left\{x: \,\, |Du(x)|< \sigma^{\frac{1+q}{p-1-q}}\right\}\cap B_r(x_0)
\end{equation}
for an appropriate constant $\hat{\mathfrak{C}}>0$. Moreover, from   \eqref{eqStabest} and \eqref{eqinclus} we obtain that
$$
   \max\left\{\mathcal{L}^N(\mathfrak{K}[v_1]\setminus \mathfrak{K}[v_2]), \mathcal{L}^N(\mathfrak{K}[v_2]\setminus \mathfrak{K}[v_1])\right\} \leq c_1 \sigma \quad \mbox{and} \quad \mathfrak{K}[v_1] \subset \left\{v_2(x)< \sigma^{\frac{p}{p-1-q}}\right\}.   
$$
In order to finish, consider $z \in \overline{\{v_1>0\}}$. Thus, by using the Non-degeneracy, Theorem \ref{LGR} we obtain that
$$
  \displaystyle \sup_{B_{c\sigma}(z)} u(x) \geq \mathfrak{c}_0(c\sigma)^{\frac{p}{p-1-q}}>\sigma^{\frac{p}{p-1-q}},
$$
for a large constant $c $. Therefore, we conclude that $z \notin \mathfrak{K}_{-c\sigma}[v_2]$.
\end{proof}

\section{A Liouville type Theorem}\label{Sec5}

Liouville type theorems are well-known in the context of elliptic PDEs and have played an important role in the modern theory of mathematical analysis due to their applications in Nonlinear equations, Free boundary problems and Differential Geometry, just to mention a few topics.

The main purpose of this subsection is to prove that a global solution to
\begin{equation} \label{ecrn}
           -\div(\Phi(  x, u, \nabla u))  +  \lambda_0(x) u_{+}^{q}(x)   = 0 \text{ in } \R^N
\end{equation}
must grow faster than $\mathfrak{C}|x|^{\frac{p}{p-1-q}}$ as $|x|\to \infty$ for a suitable constant $\mathfrak{C}>0$, unless it is identically zero.

For this end, fix $x_0 \in \R^N$, $\varsigma >0$ and $0<r_0<r$, we consider for $\rho<r$ (to be considered) the quantity $r_0=r-\rho$. Then, as in the proof of Theorem \ref{LGR}, it can be seen that the radially symmetric function $v: B_r(x_0) \to \R_+$ given by
\begin{equation} \label{radi}
 \,v(x)=\Theta(N, \lambda_0, \kappa_1, p, q) (|x-x_0|-r_0)_+^{\,\frac{p}{p-1-q}}
\end{equation}
is a weak super-solution to
\begin{equation}\label{rad eq}
   \left \{
       \begin{array}{rllll}
           -\div(\Phi(  x, u, \nabla u)) + \lambda_0(x) u_{+}^{q}(x) & = & 0 & \text{ in } & B_r(x_0) \\
           u(x) & = & \varsigma &\text{ on } & \partial B_r(x_0)
\\
           u(x) & = & 0 &\text{ in } &  \overline{B_{r_0}(x_0)}
       \end{array},
   \right.
\end{equation}
where $\Theta(N, \lambda_0, \kappa_1, p, q)$ is the biggest constant satisfying \eqref{condi} and $\rho=\left( \frac{\varsigma}{\Theta(N, \lambda_0, \kappa_1, p, q)} \right)^\frac{p-1-q}{p}$.

The explicit expression of $v$ allow us to prove the following sharp (quantitative) Liouville type result for dead core problems.

\begin{theorem}\label{Liouville} Let $u$ be a weak solution to \eqref{ecrn} with $\Phi$ as in Section \ref{Sec4}. Then, $u\equiv 0$ provided that
\begin{equation}\label{cond thm B1}
\displaystyle \limsup_{|x| \to \infty} \frac{u(x)}{|x|^{\frac{p}{p-1-q}}} < \Theta(N, \lambda_0, \kappa_1, p, q).
\end{equation}
\end{theorem}

\begin{proof}
Fixed $s_0>0$ (large enough) let us consider $w \colon  \overline{B_{s_0}} \to \mathbb{R}$ the unique (see Theorem \ref{ExisThm}) weak solution to
\begin{equation}\label{ecc1}
\left\{
\begin{array}{rclcl}
  - \div(\Phi(x, w, \nabla w)) + \lambda_0(x) w_{+}^{q}(x) & = & 0& \text{ in } & B_{s_0}\\
  w(x) & = & \sup\limits_{\partial B_{s_0}} u(x) & \text{ on } & \partial B_{s_0}.
\end{array}
\right.
\end{equation}
According to the comparison principle
$$
   u \leq w \quad \mbox{in} \quad  B_{s_0}.
$$
Moreover, due to hypothesis \eqref{cond thm B1}
\begin{equation}\label{lim rad}
\sup\limits_{\partial B_{s_0}}\frac{u(x)}{s_0^{\frac{p}{p-1-q}}} \leq \sup\limits_{ B_{s_0}}\frac{u(x)}{s_0^{\frac{p}{p-1-q}}} \leq c\Theta(N, \lambda_0, \kappa_1, p, q)
\end{equation}
for some $c \ll 1$ small enough and $s_0\gg 1$ large enough. As above, the function
\begin{equation}\label{rad eq3}
  v(x)=\Theta(N, \lambda_0, \kappa_1, p, q) \left(|x|-s_0 + \left(\frac{\sup\limits_{\partial B_{s_0}} u(x)}{\Theta(N, \lambda_0, \kappa_1, p, q)}\right)^{\frac{p-1-q}{p}} \right)_+^{\frac{p}{p-1-q}}
\end{equation}
is a weak super-solution to \eqref{ecc1}. Thus, $w\leq v$ in $B_{s_0}$.

Therefore, by  \eqref{lim rad} and \eqref{rad eq3} we conclude that
$$
   u(x) \leq \Theta(N, \lambda_0, \kappa_1, p, q)  \left(|x|- (1-c^{\frac{p-1-q}{p}})s_0 \right)_+^{\frac{p}{p-1-q}} \to 0 \quad \mbox{as} \quad s_0 \to \infty.
$$
\end{proof}

\begin{remark}
The following consequences follow from Theorem  \ref{Liouville}.
\begin{enumerate}
\item
    Note that the constant  in Theorem \ref{Liouville} is optimal (for some classes of $\Phi$) in the sense that we can not remove the strict inequality in \eqref{cond thm B1}. In fact, the function given by
$$
    u(x) = \Theta(N, \lambda_0, p, q)(|x|-r_0)_{+}^{\frac{p}{p-1-q}}
$$
solves \eqref{ecrn} (for $\Phi(x, u, \nabla u) = |\nabla u|^{p-2}\nabla u$ and $\lambda_0(x) = \lambda$) and it clearly attains the equality in \eqref{cond thm B1} for the explicit value
$$
  \Theta(N, \lambda_0, p, q)  = \left[\lambda \frac{ \left(p-1-q \right)^{p} }{ p^{p-1}(pq+N)(p-1-q )}\right]^{\frac{1}{p-1-q}}.
$$

    \item If $u$ is not a constant function, then there exist $c = c(N, p, q)>0$ and $\alpha = \alpha(N, p, q) \geq \frac{p}{p-q-1}$ such that
  $$
  \displaystyle  \mathcal{S}_r[u] \geq c.r^{\alpha}, \quad \forall \,\,R\gg 1.
    $$
    Particularly, a non-constant solution to \eqref{ecrn} must grow  at infinity at least as fast as the power $|x|^{\alpha}$ as $|x| \to \infty$.
\end{enumerate}

\end{remark}

\section{$(N-1)$-Hausdorff measure estimates for the free boundary}\label{Sec6}

In this section we focus our attention to a fine measure theoretic property of the free boundary $\partial \{u>0\} \cap \Omega$, namely, the finiteness of the corresponding Hausdorff measure (cf. Lee-Shahgholian\cite{LeeShah}). For this purpose, we restrict our analysis to solutions related to the minimization problem
\begin{equation}\tag{{\bf \text{Min}}}\label{eqMinHaus}
   \displaystyle \min_{v \in \mathcal{K}_g(\Omega)} \int_{\Omega} \mathcal{F}(x,  u, \nabla u)dx,
\end{equation}
where $\mathcal{K}_g(\Omega) = \{v \in W^{1, p}(\Omega) \suchthat v=g \,\,\, \text{on} \,\,\, \partial \Omega\}$ and
$$
   \mathcal{F}(x, u, \nabla u) = \mathcal{G}(x, \nabla u) + \mathfrak{h}(x, u).
$$
Moreover, the following assumptions must be required:
\begin{enumerate}
  \item[\checkmark] The mapping $x \mapsto \mathcal{G}(x, \xi)$ is continuous for all $\xi \in \R^N$.
  \item[\checkmark] There exists a constant $\zeta>0$ such that
    $$
    \zeta|\xi|^p \leq \mathcal{G}(x, \xi) \leq \zeta^{-1}|\xi|^p.
    $$
  \item[\checkmark] The mapping $\xi \mapsto \mathcal{G}(x, \xi)$ is differentiable with
     $$
         \Phi(x, \xi) = \nabla_{\xi} \mathcal{G}(x, \xi) \quad \text{and} \quad \frac{d}{du}\mathfrak{h}(x, u) = (q+1)\lambda(x)u_{+}^{q}.
     $$
    \item[\checkmark] The mapping $\xi \mapsto \mathcal{G}(x, \xi)$ is strictly convex.
\end{enumerate}
Under the previous structural conditions we have the existence of minimizers to \ref{eqMinHaus}. Moreover, such minimizers are non-negative solutions for the following Euler-Lagrange equation:
$$
    \div(\Phi(x, \nabla u)) = (q+1)\lambda(x)u^q(x) \quad  \text{in} \quad \{u > 0\}\cap \Omega.
$$
In other words, minimizers are weak solutions to \eqref{Maineq} with $\lambda_0(x) = (q+1)\lambda(x)$.

From now on, we will assume the hypothesis (H1)--(H4) (respectively (A1)--(A2)) on $\Phi$.

Before starting our analysis, we will show that \ref{eqMinHaus} (resp. the dead core problem \eqref{Maineq}) defines a measure supported on its free boundary.

\begin{lemma}\label{LemRadMeas} Let $\mu_{\Phi}$ be given by
$$
   \mu_{\Phi} = \div(\Phi(x, \nabla u)) - \lambda_0(x)u^{q}(x)\chi_{\{u>0\}}.
$$
Then, it defines a non-negative Radon measure supported on $\partial \{u>0\} \cap \Omega$ where $u$  minimizes  \eqref{eqMinHaus}.
\end{lemma}
\begin{proof}
Since $u$ is a non-negative minimizer to \eqref{eqMinHaus}, and $\mathcal{G}$ is strictly convex,  we have that for all $0 \leq v \in C_0^{\infty}(\Omega)$ and $\varepsilon > 0$,    the following relation is true
$$
\begin{array}{rcl}
  0 & \leq  & \displaystyle \frac{1}{\varepsilon} \int\limits_{\Omega} \left(\mathcal{G}(x, \nabla (u-\varepsilon v))- \mathcal{G}(x, \nabla u)\right)\,dx + \frac{1}{\varepsilon} \int\limits_{\Omega} \lambda(x)[(u-\varepsilon v)_{+}^{q+1}-u^{q+1}]\,dx\\
   & \leq & - \displaystyle\int\limits_{\Omega}  \Phi(x, \nabla(u-\varepsilon v)) \cdot  \nabla v \, dx + \frac{1}{\varepsilon} \int\limits_{\Omega} \lambda(x)[(u-\varepsilon v)_{+}^{q+1}-u^{q+1}] \, dx\\
   &  =  & - \displaystyle\int\limits_{\Omega}  \Phi(x, \nabla(u-\varepsilon v)) \cdot \nabla v \, dx +  \int\limits_{\{u>0\}}  \frac{\lambda(x)[(u-\varepsilon v)_{+}^{q+1}-u^{q+1}]}{\varepsilon} \, dx + \int\limits_{\{u=0\}}  \frac{\lambda(x)(-\varepsilon v)_{+}^{q+1}}{\varepsilon} \, dx\\
   & = & - \displaystyle\int\limits_{\Omega}  \Phi(x, \nabla(u-\varepsilon v)) \cdot \nabla v \, dx +  \int\limits_{\{u>0\}}  \frac{\lambda(x)[(u-\varepsilon v)_{+}^{q+1}-u^{q+1}]}{\varepsilon} \, dx .
\end{array}
$$
Finally, by taking $\varepsilon \to 0$ we obtain
$$
  \displaystyle -\int_{\Omega} \left(\Phi(x,  \nabla u)\cdot \nabla v - \lambda_0(x)u^{q}(x)\chi_{\{u>0\}}v\right)\, dx \geq 0.
$$
For this very reason, the measure given by
$$
   \displaystyle \int_{\Omega} v \, d\mu_{\Phi} = -\int_{\Omega} (\Phi(x, \nabla u)\cdot \nabla v  - \lambda_0(x)u^{q}(x)\chi_{\{u>0\}}v) \, dx
$$
is a non-negative Radon measure supported on $\partial \{u>0\} \cap \Omega$ according to Riesz's representation Theorem.
\end{proof}

In the next, we will establish upper and lower control on the $(N-1)$-Hausdorff measure of the set $\partial\{ u>0\}$ for bounded minimizers to \ref{eqMinHaus} (resp. weak solutions to \eqref{DCP1}). Such bounds imply some specific measure theoretical information about the free boundary.

\begin{theorem}\label{ThmHausd} Let $u$ be a minimizer to \ref{eqMinHaus}. Fixed $\Omega^{\prime} \Subset \Omega$ and any $x_0 \in \partial \{u>0\} \cap \Omega^{\prime}$ there exist universal constants $0< \mathfrak{c} \leq \mathfrak{C}< \infty$ such that
\begin{equation}\label{eqHausmea}
  \mathfrak{c}r^{N-1} \leq \mathcal{H}^{N-1}(B_r(x_0) \cap \partial \{u>0\}) \leq \mathfrak{C}r^{N-1}.
\end{equation}
Particularly,
$$
  \mathcal{H}^{N-1}(\partial \{u>0\} \setminus \partial_{\text{red}}\{u>0\}) = 0\footnote{The reduced free boundary $\partial_{\text{red}}\{u>0\}$ is subset of $\partial \{u > 0\}$ where there exists the normal vector in the measure theoretic sense, see \cite{EG} for a survey about geometric measure theory.}
$$
and $\partial\{u>0\}$ is locally a finite perimeter set.
\end{theorem}
\begin{proof} By using appropriated test functions $0\leq v_k \leq 1$ such that $v_k \to \chi_{B_r(x_0)}$, we can proceed with a standard approximation scheme (for almost $r > 0$). Thus, from (H3) and gradient bounds (cf. \cite{Choe1}, \cite{DB0} and \cite{Tolk}) we get using the Divergence Theorem that
$$
\begin{array}{rcl}
  \displaystyle \int\limits_{B_r(x_0)} d\mu_{\Phi} & = & \displaystyle \int\limits_{\partial B_r(x_0)} \Phi(x, \nabla u)\cdot\eta \, d\mathcal{H}^{N-1} - \int\limits_{B_r(x_0)} \lambda_0(x)u^{q}(x)\chi_{\{u>0\}}dx \\
   & \leq & \mathfrak{C}\|\nabla u\|^{p-1}_{L^{\infty}(\Omega^{\prime})}r^{N-1} \\
   & \leq & \mathfrak{C} r^{N-1}.
\end{array}
$$
This proves the upper bound in \eqref{eqHausmea} for a constant $\mathfrak{C}>0$ depending only on
$\Omega^{\prime}$, $\|u\|$ and universal parameters.

In the another hand, we may assume, without loss of generality that $x_0 = 0$. Moreover, in order to verify the lower bound, for the  sack of contradiction, let us assume  that there
exists a sequence of positive numbers such that $r_k \to 0$  as $k \to \infty$ and
\begin{equation}\label{eqHauscont}
  \mathcal{H}^{N-1}(B_{r_k}(x_0) \cap \partial \{u>0\}) = o\left(r_k^{N-1}\right).
\end{equation}
Now, by defining
$$
   v_k(x) \defeq \frac{u(r_kx)}{r_k^{\frac{p}{p-1-q}}}
$$
we obtain the sequence of non-negative measures $\mu_{\Phi_k}$ defined in $B_{\frac{3}{4}}$, given  by
$$
   \mu_{\Phi_k} \defeq \div (\Phi_k(x, \nabla v_k)) - \lambda_0^k(x)v_k^q\chi_{\{v_k>0\}},
$$
where
$$
  \Phi_k(x, \xi) \defeq r_k^{-\frac{(1+q)(p-1)}{p-1-q}}\Phi\left(x_0 + r_kx, r_k^{\frac{1+q}{p-1-q}}\xi\right) \quad \text{and} \quad \lambda_0^k(x) \defeq r_k^{\frac{pq}{p-1-q}}\lambda_0(x_0 + r_kx).
$$
As before, it is easy to check that $\Phi_k(x, \xi)$ satisfies (H1)--(H4) (resp. (A1)--(A2)).
Now, we may assume, via compactness, that $\mu_{\Phi_k} \rightharpoonup \mu_0$ in the sense of measures. Furthermore, \eqref{eqHauscont} implies that
\begin{equation}\label{eqQlim}
  \mu_{\Phi_k} \rightharpoonup 0.
\end{equation}
Next, we will show that
\begin{equation}\label{eqQlim1}
  \mu_0  = \div(\Phi_0(x_0, \nabla u_0)),
\end{equation}
where, up to a subsequence $\displaystyle \Phi_0  = \lim_{k \to \infty} \Phi_k$ and $\displaystyle u_0(x) = \lim_{k \to \infty} u_k(x)$.
Recall that from Corollary \ref{CorPor}, it holds that $\mathcal{L}^N(\partial \{u_0> 0\}) = 0$. For this reason, we will only check \eqref{eqQlim1} for balls contained in $\{u_0>0\}$ and in $\{u_0 = 0\}$.
First, let us consider a ball $B \subset \{u_0>0\}$. Notice that from the growth estimates near free boundary points, Corollary \ref{GradEst} and property (H3), we have
\begin{equation}\label{eqestUpsi}
    |\Phi_k(x, \nabla v_k)| \leq \mathfrak{C}(N, p)\|\nabla v_k\|_{L^{\infty}\left(B_{3/4}\right)}^{p-1} \leq \mathfrak{C}(N, p, q, \lambda_{+}).
\end{equation}
Hence, up to a subsequence,  we have
$$
   \Phi_k \to \Phi_0
$$
in the \text{weak*} topology in $L^{\infty}\left(B_{\frac{3}{4}}\right)$. Moreover,  $u_j \to u_0$ locally uniformly in the $C^1$-topology in $B$. Consequently,
$$
  \Phi_k \to \Phi_0(x_0, \nabla u_0)
$$
in the \text{weak}* topology in $L^{\infty}(B)$, where we have used \eqref{eqestUpsi}. On the other hand,
$$
  \displaystyle \int_{B} \lambda_0^k(x)v_k^q\chi_{\{v_k>0\}} dx \leq \mathfrak{C}(N, p, q, \lambda_{+}) \mathcal{L}^N(B)r_k^{\frac{pq}{p-1-q}} \to 0 \quad \text{as} \quad k \to \infty,
$$
where we have used Theorem \ref{MThm}.
Therefore, we have proved \eqref{eqQlim1} for this first case.

Now, consider $B \subset \{u_0 = 0\}$.
It is immediate that
$$
   \div(\Phi_0(x_0, \nabla u_0))(B) = 0.
$$
Moreover, if $B_j$ is a sequence of balls such that $B_j \nearrow B$, then for
some $k_j \in \mathbb{N}$ we get
\begin{equation}\label{eqnull}
  u_k \equiv 0 \quad \text{in} \quad B_j \quad \forall \,\, k >k_j.
\end{equation}
In fact, let $\hat{B} \subset B$ and suppose that there were a subsequence $u_{k_j}$ fulfilling $u_{k_j} \neq 0$ in $\hat{B}$. Then, according to the strong non-degeneracy given in Theorem \ref{nondeg est}, there must exist points $Y_{k_j} \in \hat{B}$ such that
$$
  u_{k_j}(Y_{k_j}) \geq \mathfrak{c}>0.
$$
Thus, we may assume, passing a subsequence if necessary, $Y_{k_j} \to Y_0 \in \hat{B}$. Moreover, as $u_{k_j} \to u_0$ we obtain, after passing to the limit that $u_0(Y_0)>0$, which is a contradiction to our assumption.
Therefore, from \eqref{eqnull} we obtain that
$$
  \mu_{\Phi_k} \to 0.
$$
Thus \eqref{eqQlim1} is checked for the second case.

Finally, by combining \eqref{eqQlim} and \eqref{eqQlim1} we obtain
$$
\left\{
\begin{array}{rclcl}
  \div(\Phi_0(x_0, \nabla u_0)) & = & 0 & \text{in}& \Omega \\
  u_0(x_0) & = & 0, & &
\end{array}
\right.
$$
and the Strong Maximum Principle from \cite{Vaz} implies that $u_0 \equiv 0$. Nevertheless, as before, we obtain a contradiction with the non-degeneracy of $u_0 \geq \mathfrak{c}>0$ given in Theorem \ref{nondeg est}. Such a contradiction proves the result.
\end{proof}

\vspace{0.2cm}

\begin{remark}\label{Remarkfinal}
  As an immediate consequence of previous estimates we conclude that $\mathfrak{F}(u)$ has locally finite perimeter. Moreover, the reduced free boundary $\mathfrak{F}_{\text{red}}(u) \defeq \partial_{\textrm{red}} \{u_0 >0\}$ has a total $\mathcal{H}^{N-1}$ measure in the sense that $\mathcal{H}^{N-1}(\mathfrak{F}(u) \setminus \mathfrak{F}_{\text{red}}(u)) =0$. Particularly, the free boundary has an outward vector for $\mathcal{H}^{N-1}$ almost everywhere in $\mathfrak{F}_{\text{red}}(u)$ (cf. \cite{EG} for more details).
\end{remark}

\subsection*{Acknowledgments}
\hspace{0.6cm} This work has been partially supported by Consejo Nacional de Investigaciones Cient\'{i}ficas y T\'{e}cnicas (CONICET-Argentina) PIP 11220150100036CO. J.V. da Silva and A. Salort are members of CONICET.

\end{document}